\documentclass[12pt]{article}
\usepackage{amsfonts,epsf,amsmath,amssymb,graphicx,amsthm,fullpage}

\usepackage{graphicx}
\usepackage{subfigure}
\usepackage{caption}

\newtheorem{theorem}{Theorem}[section]
\newtheorem{corollary}[theorem]{Corollary}
\newtheorem{lemma}[theorem]{Lemma}
\newtheorem{proposition}[theorem]{Proposition}
\newtheorem{conjecture}[theorem]{Conjecture}
\newtheorem{problem}[theorem]{Problem}
\newtheorem{claim}[theorem]{Claim}
\theoremstyle{definition}
\newtheorem{definition}[theorem]{Definition}

\newtheorem{example}[theorem]{Example}
\newtheorem{question}[theorem]{Question}
\newtheorem{observation}[theorem]{Observation}
\numberwithin{equation}{section}

\def\cF{{\mathcal F}}

\def\NN{{\mathbb N}}
\def\esub{\subseteq}
\def\VEC#1#2#3{#1_{#2},\ldots,#1_{#3}}

\def\Gb{\overline{G}}
\def\Hb{\overline{H}}
\def\Kb{\overline{K}}

\def\cD{\mathcal{D}}
\def\FL#1{\left\lfloor #1\right\rfloor}
\def\CL#1{\left\lceil #1\right\rceil}
\def\FR#1#2{\frac{#1}{#2}}
\def\CH#1#2{\binom{#1}{#2}}

\def\SE#1#2#3{\sum_{{#1}={#2}}^{#3}}
\def\PE#1#2#3{\prod_{{#1}={#2}}^{#3}}

\def\C#1{\left| #1\right|}

\def\er{{\rm e}}
\def\bG{{\bf G}}
\def\bT{{\bf T}}

\begin{document}

\title{On reconstruction of graphs from the multiset of subgraphs obtained by
deleting $\ell$ vertices}
\author{
Alexandr V. Kostochka\thanks{University of Illinois at Urbana--Champaign,
Urbana IL 61801, and Sobolev Institute of Mathematics, Novosibirsk 630090,
Russia: \texttt{kostochk@math.uiuc.edu}.  Research supported in part by NSF
grants DMS-1600592 and grants 18-01-00353A and 19-01-00682  of the Russian
Foundation for Basic Research.}\,,
Douglas B. West\thanks{
Departments of Mathematics, Zhejiang Normal University, Jinhua 321004, China,
and University of Illinois, Urbana IL 61801, U.S.A.; west@math.uiuc.edu.
Research supported by Recruitment Program of Foreign Experts, 1000 Talent Plan,
State Administration of Foreign Experts Affairs, China.} 
}

\date{
{\it Dedicated to the memory of Vladimir Levenshtein}\\\smallskip
}
\maketitle

\vspace{-2pc}

\begin{abstract}
The Reconstruction Conjecture of Ulam asserts that, for $n\geq 3$, every
$n$-vertex graph is determined by the multiset of its induced subgraphs with
$n-1$ vertices.  The conjecture is known to hold for various special classes of
graphs but remains wide open.  We survey results on the more general conjecture
by Kelly from 1957 that for every positive integer $\ell$ there exists $M_\ell$
(with $M_1=3$) such that when $n\geq M_\ell$ every $n$-vertex graph is
determined by the multiset of its induced subgraphs with $n-\ell$ vertices.

\end{abstract}

\baselineskip16pt

\section{Introduction} \label{sec:intro}
Among the basic problems in combinatorics are so-called {\em reconstruction
problems}, where we try to identify an object knowing only partial information
about it.  A classic example is the famous Graph Reconstruction Conjecture of
Ulam first posed in the thesis of Kelly in 1942:

\begin{conjecture}[The Reconstruction Conjecture; Kelly~\cite{kel1,kel2},
Ulam~\cite{U}]
For $n\geq 3$, every $n$-vertex graph is determined by the multiset of its
$(n-1)$-vertex induced subgraphs.
\end{conjecture}

The multiset of $(n-1)$-vertex induced subgraphs is called the {\it deck} of
the graph, with each such subgraph being a {\it card} in the deck.  The vertices
in the cards are unlabeled, meaning that only the isomorphism class of each
card is given.  The restriction $n\geq 3$ is needed because the two graphs with
two vertices have the same deck.  A graph is {\it reconstructible} if it is
determined by its deck, meaning that no other graph has the same deck.  In this
terminology, the Reconstruction Conjecture asserts that every graph with at
least three vertices is reconstructible.

The conjecture has attracted a lot of attention.  Graphs in many families are
known to be reconstructible; these include disconnected graphs, trees, regular
graphs, and perfect graphs.  Surveys on graph reconstruction
include~\cite{Bondy91,BH77,Lauri97,LS,Maccari02}.

Various parameters have been introduced to measure the difficulty of 
reconstructing a graph.  Harary and Plantholt~\cite{HP} defined the
{\it reconstruction number} of a graph to be the minimum number of cards from
its deck that suffice to determine it, meaning that no other graph has this
multiset of cards in its deck (surveyed in~\cite{AFLM}).  All trees with at
least five vertices have reconstruction number $3$ (Myrvold~\cite{Myrvold90}),
and almost all graphs have reconstruction number $3$
(Bollob\'as~\cite{Bollobas}).  No graphs have reconstruction number $2$, since
two cards cannot determine whether the two vertices deleted to form the cards
are adjacent.

Let $K_{\VEC t1r}$ denote the complete $r$-partite graph with part-sizes
$\VEC t1r$.  Since $K_{t,t}$ and $K_{t+1,t-1}$ have $t+1$ common cards, the
reconstruction number of an $n$-vertex graph can be as large as $\FR n2+2$
(Myrvold~\cite{Myrvold89}).  (Here $K_{s,t}$ is the complete bipartite graph
with parts of sizes $s$ and $t$.)  Harary and Plantholt~\cite{HP} strengthened
the Reconstruction Conjecture by conjecturing that when $n\ge3$ every
$n$-vertex graph has reconstruction number at most $\FR n2+2$, with equality
only for $K_{n/2,n/2}$ and $2K_{n/2}$ in general, plus $P_4$.
Kocay and Kreher~\cite{KK} constructed $n$-graphs with reconstruction number
$\FR n2+1$ when $n=4q-4$ and $q$ is a prime power congruent to $1$ modulo $4$.

We can also study the reconstruction number of graph properties.
Myrvold~\cite{Mthe} and Bowler et al.~\cite{BBFM} showed that any $\FL{n/2}+2$
cards determine whether an $n$-vertex graph is connected.  Much effort went
into reducing the number of cards needed to determine $m$, the number of edges.
Myrvold~\cite{Myrvold92} showed that $m$ and in fact also the degree list are
determined by any $n-1$ cards when $n\ge7$ (this is sharp).  Monikandan and
Balakumar~\cite{MB} showed that $m$ is determined within $1$ by any $n-2$ cards
(strengthening~\cite{RM}).  Woodall~\cite{Woodall} proved for
$n\ge\max\{34,3p^2+1\}$ that $m$ is determined within $p-2$ by $n-p$ cards.
Brown and Fenner~\cite{BF} proved that $m$ is determined by any $n-2$ cards
when $n\ge29$, and they presented two $8$-vertex graphs with six common cards
whose numbers of edges differ by $1$.  Groenland, Guggiari, and Scott~\cite{GGS}
proved that in fact $m$ is determined by any $n-\sqrt n\,/20$ cards when
$n$ is sufficiently large.

These results concern not so much the reconstruction number as the
{\it adversary reconstruction number}~\cite{Mthe} or {\it universal
reconstruction number}~\cite{BBF}, which is the minimum $t$ such that any $t$
cards determine the graph (or a particular property).  Bowler, Brown, and
Fenner~\cite{BBF} presented infinite families of pairs of graphs sharing
$2\FL{(n-1)/3}$ common cards, improving Myrvold~\cite{Mthe,Myrvold90}.
They conjectured that when $n$ is sufficiently large, every $n$-vertex graph
is determined by any $2\FL{(n-1)/3}+1$ of its cards ($n>12$ is needed).

Kelly~\cite{kel2} took another direction, considering cards obtained by
deleting more vertices.

\begin{definition}
A {\it $k$-card} of a graph is an induced subgraph having $k$ vertices.  The
\emph{$k$-deck} of $G$, denoted $\cD_k(G)$, is the multiset of all $k$-cards
(given as isomorphism classes).  A graph $G$ is {\it determined by its $k$-deck}
if $\cD_k(H)=\cD_k(G)$ implies $H\cong G$.  A graph $G$ (or a graph invariant)
is {\it $\ell$-reconstructible} if it is determined by $\cD_{|V(G)|-\ell}(G)$
(agreeing on all graphs having that deck).  The
{\it maximum reconstructibility} of a graph $G$
is the maximum $\ell$ such that $G$ is $\ell$-reconstructible.
\end{definition}

Study of reconstruction from the $k$-deck was begun by Manvel~\cite{Manvel}.
There followed several papers by N\'ydl, including surveys (\cite{N84,N01}) of
the early results.  N\'ydl studied the least $k$ (as a function of $n$) such
that every $n$-vertex graph (or every $n$-vertex graph in a restricted family
such as trees) is determined by its $k$-deck.

For an $n$-vertex graph, ``determined by its $k$-deck'' and
``$\ell$-reconstructible'' have the same meaning when $k+\ell=n$.
The motivation for defining the maximum reconstructibility as a measure
of the ease of reconstructing a graph is the following elementary observation.

\begin{observation}\label{k-1}
For any graph $G$, the $k$-deck $\cD_k(G)$ determines the $(k-1)$-deck
$\cD_{k-1}(G)$.
\end{observation}

The proof is that each card in $\cD_{k-1}(G)$ appears in $\C{V(G)}-k+1$
cards in $\cD_k$.  By Observation~\ref{k-1}, information that is determined by
the $k$-deck is also determined by the $j$-deck when $j>k$.  This leads to a
stronger version of the Reconstruction Conjecture.

\begin{conjecture}[Kelly~\cite{kel2}]
For $\ell\in\NN$, there is an integer $M_\ell$ such that every graph with at
least $M_\ell$ vertices is $\ell$-reconstructible.
\end{conjecture}

The original Reconstruction Conjecture is the claim $M_1=3$.
Having checked by computer that every graph with at least six and at most nine
vertices is $2$-reconstructible (there are $5$-vertex graphs that are not),
McMullen and Radziszowski~\cite{MR} asked whether $M_2=6$.  With computations
up to nine vertices, Rivshin and Radziszowski~\cite{RR} conjectured
$M_\ell\le3\ell$.  N\'ydl~\cite{N92} disproved this, showing that $M_\ell$
must grow at least superlinearly in $\ell$; that is, $M_\ell/\ell\to\infty$.
He proved that for any $n_0\in\NN$ and $0<q<1$, there are nonisomorphic
$n$-vertex graphs for some $n$ larger than $n_0$ having the same
$\FL{qn}$-deck.

For more detailed understanding, it is natural to study the threshold number
of vertices for $\ell$-reconstructibility for graphs in a given family, which
may be smaller than for the family of all graphs.  We may write this as
$M_\ell(\bG)$ for a family $\bG$.  For example, Spinoza and West~\cite{SW} (see
Section~\ref{sec:maxdeg2}) proved that the path $P_{2\ell}$ and the graph
$C_{\ell+1}+P_{\ell-1}$ have the same $\ell$-deck (here ``$+$'' denotes
disjoint union of graphs, while $C_n$, $P_n$, $K_n$ respectively denote the
cycle, path, and complete graph on $n$ vertices).
Thus $M_\ell(\bG_2)\ge 2\ell+1$ for $\ell\ge3$, where $\bG_d$ is the family
of graphs with maximum degree at most $d$.  In fact, they proved
$M_\ell(\bG_2)=2\ell+1$.

For the family $\bT$ of trees, the same lower bound is known and is sharp
for $\ell=2$ (Giles~\cite{Giles}), but equality is open for $\ell\ge3$.  Let
$S_{a,b,c}$ be the subdivision of $K_{1,3}$ consisting of paths of lengths $a$,
$b$, and $c$ with one common endpoint (in general, a tree consisting of paths
with one common endpoint is called a ``spider'').
N\'ydl~\cite{N81} observed that $S_{k-1,k-1,1}$ and $S_{k,k-2,1}$ are spiders
with $2k$ vertices having the same $k$-deck.  We will give a short proof of
this using the results on common $k$-decks for graphs in $\bG_2$ that are
discussed in Section~\ref{sec:maxdeg2}.  The result implies
$M_\ell(\bT)\ge2\ell+1$, and N\'ydl~\cite{N81} conjectured that equality holds.
One can generalize this question to the family $\bT_r$ of connected graphs $G$
such that $\C{E(G)}-\C{V(G)}+1\le r$; that is, $\bT=\bT_0$.  N\'ydl~\cite{N81}
constructed two graphs with $3k+9$ vertices and $3k+12$ edges having the same
$2k$-deck, thus yielding $M_\ell(\bT_4)\ge 3\ell-18$.

As with ordinary reconstruction, proving that the graphs in a family $\bG$ are
$\ell$-reconstruct\-ible may involve two steps.  One is to show that the
family is {\it $\ell$-recognizable}, meaning that whether $G\in\bG$ holds is
determined by $\cD_{|V(G)|-\ell}$.  That is, every graph having the same
deck as a graph in $\bG$ is also in $\bG$.  For example, Manvel~\cite{Manvel}
showed that when $|V(G)|=n\ge6$, the $(n-2)$-deck determines whether $G$ is
connected, acyclic, unicyclic, regular, or bipartite.  That is, these
properties are $2$-reconstructible when $n\ge6$.

The pair $\{P_{2\ell},C_{\ell+1}+P_{\ell-1}\}$ mentioned earlier shows for
$n$-vertex graphs that guaranteeing $\ell$-reconstructibility of the property
of connectedness (or $\ell$-recognition of the family of connected graphs)
requires $n\ge2\ell+1$.  The correct general threshold remains open.

On the other hand, the fraction of $n$-vertex graphs whose maximum
reconstructibility is at least $(1-o(1))n/2$ tends to $1$ (see
Section~\ref{sec:almostall}).  This was observed originally by
M\"uller~\cite{Muller}.  In particular, there is surprisingly small difference
between the maximum reconstructibility of almost all graphs and the failure of
reconstructibility of the property of connectedness.
Spinoza and West~\cite{SW} showed that in fact in this setting only
$\CH{\ell+2}2$ cards are needed, generalizing the concept of reconstruction
number to $\ell$-reconstruction number.

For some easily reconstructed families it is natural to fix the number of
vertices kept in each card.  The $2$-deck of $G$ determines only $|E(G)|$ and
$|V(G)|$.  The $3$-deck determines also the number of edge incidences, whether
$G$ is triangle-free, and whether $G$ belongs to the family of complete
multipartite graphs, since that is true if and only if $P_2+P_1$ is not an
induced subgraph.

Results on $\ell$-reconstructibility are known for degree lists, connectedness,
trees, graphs with maximum degree $2$, random graphs, and graphs that are
disconnected, complete multipartite, or regular.  We describe these results in
the subsequent sections, and we include a few new results about these classes.
In Section~\ref{sec:trees} we offer a new short proof of N\'ydl's result that
$M_\ell(\bT)\ge 2\ell+1$.
In Section~\ref{sec:rpartite} we show that $n$-vertex graphs whose components
have at most $n-\ell$ vertices are $\ell$-reconstructible, while graphs with
components having more vertices are guaranteed to be $\ell$-reconstructible
only if the original Reconstruction Conjecture is true.
In Section~\ref{sec:regular} we show that $r$-regular graphs with connectivity
$1$ are $(r+1)$-reconstructible.

We mention two other models of reconstruction.  Levenshtein et al.~\cite{LKKM}
considered a local version of reconstruction in which the vertices of a graph
are labeled and for an $n$-vertex graph $G$ we have only $n$ cards: for each
vertex $v$, we are given the set $B_2(v)$ of the vertices at distance at most
$2$ from $v$ (but do not know which of them are adjacent to $v$). It was proved
in~\cite{LKKM} that every connected graph whose girth is at least $7$ and whose
diameter and minimum degree are at least $2$ is reconstructible in this model.
The authors also provided a graph with girth $6$ that is not reconstructible.
Levenshtein~\cite{Lev,Le} posed a more general problem where we know the sets
$B_t(v)$ instead of $B_2(v)$ and presented partial results on it.

In Section~\ref{sec:relation}, we describe the main results on a model 
that uses the term ``$k$-reconstructible'' with a different meaning.
In that model, due to Fra\"iss\'e~\cite{Fr}, we are reconstructing digraphs
(viewed as general binary relations), and we are told the identities of the
deleted vertices.  Our aim in describing these alternative models is to avoid
future confusion.

\section{Degree Lists}\label{sec:deglist}

The {\it degree list} (also called {\it degree sequence}) of a graph is the
multiset of degrees of the vertices in the graph.  When studying the
Reconstruction Problem, the first thing one learns is that the degree list of 
a graph is $1$-reconstructible.  It suffices to find the number $m$ of edges,
because the degree of each vertex $v$ is the difference between $m$ and the
number of edges in the card $G-v$.  The number $m$ is the information provided
by the $2$-deck (known by Observation~\ref{k-1}); one can also compute $m$
directly from the $(n-1)$-deck by $m=\sum_v |E(G-v)|/(n-2)$.

For decks with smaller cards, reconstruction of anything becomes more
difficult.  The pairs of graphs with the same $k$-decks mentioned in the 
introduction have the same degree list, but it is easy to construct examples
with different lists.

\begin{example}\label{list}
{\it For any positive $a,b,c$ with $a+b+c=t\ge4$, the graphs $C_t+P_1$ and
$S_{a,b,c}$ with $t+1$ vertices all have the same $3$-deck.}  Note that
$\Delta(C_t+P_1)=2$ and $\Delta(S_{a,b,c})=3$.  All these graphs have $t$
copies of $P_3$ and $t(t-3)$ copies of $P_2+P_1$ as induced subgraphs (this
involves a few cases for $S_{a,b,c}$), and the remaining $3$-vertex induced
subgraphs all have no edges.  Hence $\cD_3(C_t+P_1)=\cD_3(S_{a,b,c)})$,
for all such $(a,b,c)$.
\end{example}

Concerning thresholds, this easy example shows that guaranteeing
$\ell$-reconstructibility of the degree list or the maximum degree
requires $n>\ell+3$.  In this example, we considered $\cD_k(G)$ with
$k\in\{\Delta(G),\Delta(G)+1\}$.  Manvel observed that having one more vertex
in the cards prevents such examples.

\begin{theorem}[Manvel~\cite{Manvel}]\label{deltadeg}\label{tlist}
The degree list of a graph $G$ is determined by $\cD_{\Delta(G)+2}(G)$.
\end{theorem}

The result for $\Delta(G)$ is easy since all induced subgraphs with at most
$k$ vertices are visible in the $k$-deck.  Hence $\cD_{\Delta(G)+2}(G)$ shows
that $G$ has a vertex of degree $\Delta(G)$ and none larger.  To determine the
degree list, one can then proceed by induction on $r$ to count the vertices of
degree $\Delta(G)-r$ using the following tool observed originally by
Manvel~\cite{Manvel}.

\begin{lemma}\label{count}
Let $G$ be an $n$-vertex graph.  The sum, over all cards in
the $k$-deck $\cD_k(G)$, of the number of vertices of degree $j$ in the card,
equals $\SE ij{j+n-k} a_i\CH ij\CH{n-1-i}{k-1-j}$, where $a_i$ is the number of
vertices having degree $i$ in $G$.
\end{lemma}

The lemma holds because a vertex $v$ appears and has degree $j$ in a card
in the $k$-deck if and only if the card is formed by choosing $v$ along with
$j$ of its neighbors and $k-1-j$ of its nonneighbors.  The lemma yields the
following corollary by solving for $a_i$ through successively smaller $i$.

\begin{corollary}[Manvel~\cite{Manvel}]\label{klist}
The degree list of a graph $G$ is determined when both $\cD_k(G)$ and the
numbers of vertices with degree $i$ for all $i$ at least $k$ are known.
\end{corollary}

\begin{example}
For sharpness of Theorem~\ref{deltadeg}, Manvel~\cite{Manvel} showed that the
maximum degree itself is not always determined by $\cD_{\Delta(G)+1}(G)$.
He constructed graphs $G$ and $H$ such that $\Delta(G)=k$, $\Delta(H)=k-1$, and
$\cD_k(G)=\cD_k(H)$.  Both graphs are forests of stars.  However, in
this construction the number of vertices is $(k+2)2^{k-2}$, exponential in $k$.
In particular, $G=\sum_i\CH{k}{2i}K_{1,k-2i}$ and
$H=\sum_i\CH{k}{2i+1}K_{1,k-1-2i}$.  From this Manvel concluded that
for all $k$ there exist nonisomorphic graphs with the same $k$-deck.
\end{example}

\begin{question}\label{deltaalt}
What is the smallest value of $n$ such that $n$-vertex graphs $G$ and $H$
with maximum vertex degrees $k$ and $k-1$ exist having the same $k$-deck?
\end{question}

Lemma~\ref{count} and Corollary~\ref{klist} are used as tools in 
reconstruction of degree lists.

\begin{theorem}[Chernyak~\cite{Che}]\label{Chernyak}
When $n\ge6$, the degree list of every $n$-vertex graph is $2$-reconstructible.
\end{theorem}

This result is sharp, because the $5$-vertex graphs $C_4+P_1$ and $S_{2,1,1}$
from Example~\ref{list} have the same $3$-deck but different degree lists.
Exact results were pushed one step further.

\begin{theorem}[Kostochka--Nahvi--West--Zirlin~\cite{KNWZ}]\label{deg3rec}
For $n\ge7$, the degree list of every $n$-vertex graph is $3$-reconstructible.
\end{theorem}

Theorem~\ref{deg3rec} is sharp: the $6$-vertex graphs $C_5+P_1$ and
$S_{2,2,1}$ and $S_{3,1,1}$ from Example~\ref{list} all have the same $3$-deck.
Note that since the $(n-2)$-deck determines the $(n-3)$-deck,
Theorem~\ref{deg3rec} combined with an analysis of $6$-vertex graphs implies
Theorem~\ref{Chernyak}.

By making more thorough use of Lemma~\ref{count}, Taylor~\cite{Taylor} obtained
a surprisingly small general threshold on the number of vertices for
$\ell$-reconstructibility of the degree list.

\begin{theorem}[Taylor~\cite{Taylor}]\label{Taylor}
For $n\ge g(\ell)$, the degree list of every $n$-vertex graph is 
$\ell$-reconstructible, where
$$
g(\ell) = (\ell-\log{\ell}+1)
\left( \er + \frac{\er \log{\ell} + \er +1}{(\ell-1) \log{\ell}-1} \right) +1.
$$
Here $\er$ denotes the base of the natural logarithm.
\end{theorem}

This result also shows that the degree list of an $n$-vertex graph is
reconstructible from the $k$-deck when $n$ is not too much larger than $k$,
regardless of the value of the maximum degree.  In particular,
$n\ge\frac{1+o(1)}{1-1/\er}k$ suffices.  This theorem seems rather strong
about reconstructibility but perhaps does not say much about
Question~\ref{deltaalt}, giving only a linear lower bound.

Theorem~\ref{Taylor} is strong but likely not sharp; answering the next
question would improve it and generalize Theorem~\ref{deg3rec}.
By Theorem~\ref{Taylor} the threshold is asymptotically at most $\er\ell$.
One could begin by seeking the largest graphs whose degree lists are not
$4$-reconstructible.

\begin{question}
For fixed $\ell\in\NN$, what is the least threshold $n_\ell$ such that the
degree list of every graph with at least $n_\ell$ vertices is
$\ell$-reconstructible?
\end{question}

\section{Connectedness}\label{sec:connect}
For graphs with at least three vertices, connectedness is $1$-reconstructible,
because an $n$-vertex connected graph has at least two connected $(n-1)$-cards,
while a disconnected graph has at most one connected $(n-1)$-card (when
$n\ge3$).  Manvel~\cite{Manvel} strengthened this result.

\begin{theorem}[Manvel~\cite{Manvel}]\label{thm: two deck}\label{manvel}
For $n \ge 6$, the connectedness of an $n$-vertex graph is $2$-reconstructible,
and the threshold for $n$ is sharp.
\end{theorem}

The threshold is sharp because the $5$-vertex graphs $C_4+P_1$ and $S_{2,1,1}$
have the same $3$-deck (Example~\ref{list}).  In fact, these graphs and their
complements are the only $5$-vertex graphs that are not
$2$-reconstructible~\cite{MR}.

All results that obtain a function $f(\ell)$ such that some property or class
of graphs is $\ell$-reconstructible for graphs with at least $f(\ell)$ vertices
provide support for Kelly's Conjecture.  For general $\ell$, this is known for
connectedness.

\begin{theorem}[Spinoza--West~\cite{SW}]\label{conn}
For $\ell\in\NN$, the connectedness of every $n$-vertex graph is 
$\ell$-reconstructible when $n>2\ell^{(\ell+1)^2}$
\end{theorem}

The threshold in Theorem~\ref{conn} is not sharp.

\begin{conjecture}[Spinoza--West~\cite{SW}]\label{connconj}
For $n\ge 2\ell+2$, the connectedness of an $n$-vertex graph is
$\ell$-reconstructible, and the threshold for $n$ is sharp.
\end{conjecture}

For $\ell=2$, the $5$-vertex graphs $C_4+P_1$ and $S_{2,1,1}$ again show that
the conjecture is sharp.  For larger $\ell$ the right answer may be $2\ell+1$,
which is needed due to the example $\{P_{2\ell},C_{\ell+1}+P_{\ell-1}\}$
(see Section~\ref{sec:maxdeg2}).  There are three sets of two $7$-vertex graphs
that have the same $4$-deck; none consists of a connected and a disconnected
graph.  N\'ydl's graphs showing that $M_\ell$ grows superlinearly are all
connected.  We believe that connectedness of an $n$-vertex graph is
$\ell$-reconstructible whenever $\ell<\CL{n/2}$ (except for
$\{C_4+P_1,S_{2,1,1}\}$).

For $\ell=3$, Spinoza and West~\cite{SW} improved the threshold in
Theorem~\ref{conn} to $n\ge25$.  The exact answer was found later.

\begin{theorem}[Kostochka--Nahvi--West--Zirlin~\cite{KNWZ}]\label{conn3}
For every graph with at least seven vertices, connectedness is
$3$-reconstructible.
\end{theorem}

Theorem~\ref{conn3} is sharp due to $\{C_5+P_1,S_{2,2,1},S_{3,1,1}\}$
(Example~\ref{list}).  When combined with a short analysis of $6$-vertex
graphs, Theorem~\ref{conn3} implies Theorem~\ref{manvel}.  The proof of 
Theorem~\ref{conn3} uses Theorem~\ref{deg3rec} to reduce the problem to graphs
with exactly two vertices of degree $1$ and none of degree $0$.

Toward Conjecture~\ref{connconj}, it would be interesting to find a substantial
improvement of the threshold in Theorem~\ref{conn} or to find the largest two
graphs, one connected and one disconnected, that have the same deck of
subgraphs with four vertices deleted.

\section{Graphs with Maximum Degree $2$}\label{sec:maxdeg2}

In Problem 11898 of the American Mathematical Monthly, Stanley posed a
question related to reconstructing $2$-regular graphs from their $k$-decks.

\begin{problem}[Stanley \cite{Stanley}]\label{stanley}
Let $n$ and $k$ be integers, with $n\ge k \ge 2$.  Let $G$ be a graph with $n$
vertices whose components are cycles of length greater than $k$.  Let $i_k(G)$
be the number of $k$-element independent sets of vertices of $G$. Show that
$i_k(G)$ depends only on $k$ and $n$.  
\end{problem} 

Let $s(G,H)$ denote the number of induced subgraphs of $G$ isomorphic to $H$.
Stanley's problem asserts $s(G,\overline{K}_k)=s(G',\overline{K}_k)$ for
$n$-vertex $2$-regular graphs $G$ and $G'$ whose components have length greater
than $k$ (here $\Hb$ denotes the complement of $H$).  Stanley's proposed
solution of Problem~\ref{stanley} used generating functions.

Independent sets are just one type of $k$-vertex induced subgraph.  In a graph
with maximum degree $2$ whose cycles have more than $k$ vertices, all
$k$-vertex induced subgraphs are {\it linear forests}, meaning disjoint unions
of paths.  By looking at a larger class of graphs, Spinoza and West gave
a bijective proof by induction on $k$ that proves the same conclusion for
the number of subgraphs isomorphic to any $k$-vertex linear forest and thereby
proves that the graphs with the stated property all have the same $k$-deck.

\begin{theorem}[Spinoza--West~\cite{SW}]\label{main}
Let $G$ and $G'$ be graphs with maximum degree $2$ having the same number
of vertices and the same number of edges.  If every component in each graph is
a cycle with at least $k+1$ vertices or a path with at least $k-1$ vertices,
then $\cD_k(G)=\cD_k(G')$.
\end{theorem}

Important cases of the theorem, and indeed its proof, are captured by the
following three statements, among which the third is the key, proved
inductively.

\begin{claim}\label{twocomp}

$\cD_k(C_{q+r})=\cD_k(C_q+C_r)$ if $q,r\ge k+1$,

$\cD_k(P_{q+r})=\cD_k(C_q+P_r)$ if $q\ge k+1$ and $r\ge k-1$, and

$\cD_k(P_{q-1}+P_r)=\cD_k(P_q+P_{r-1})$ if $q,r\ge k$.
\end{claim}

The proof of Theorem~\ref{main} reduces to Claim~\ref{twocomp}
by the following natural lemma.

\begin{lemma}\label{samedeck}
If $G$, $G'$, and $H$ are graphs, then $\cD_k(G)=\cD_k(G')$ if and only if
$\cD_k(G+H)=\cD_k(G'+H)$.
\end{lemma}

For every graph $G$ with maximum degree $2$, Theorem~\ref{main} provides
a lower bound for the value $k$ such that $\cD_k(G)$ determines $G$ and
hence an upper bound on the maximum reconstructibility.  Except for some
small instances, these bounds turn out to be sharp.  Without giving the
complete details of the statement, the result is the following.

\begin{theorem}\label{rho2}
Let $G$ be a graph with maximum degree $2$.  If $m$ is the maximum number of
vertices in a component, $F$ is a component with $m$ vertices,
and $m'$ is the maximum number of vertices in a component other than $F$,
then $G$ is $k$-deck reconstructible if and only if
$k\ge\max\{\FL{m/2}+\epsilon,m'+\epsilon'\}$, where $\epsilon=1$ if $F$ is
a path (otherwise $\epsilon=0$), and $\epsilon'\in\{0,1,2\}$.
\end{theorem}

We omit the technical definition of $\epsilon'$ that incorporates the 
small exceptions to the general formula.  In particular, for a $2$-regular
$n$-vertex graph $G$, the formula for $k$ simplifies to $\max\{\FL{m/2},m'\}$. 
Thus the maximum reconstructibility of the cycle $C_n$ is $\CL{n/2}$, and no
graph in this class has smaller reconstructibility.

That is, every $2$-regular $n$-vertex graph is $\CL{n/2}$-reconstructible,
and in fact graphs with maximum degree $2$ are $\FL{n/2}$-reconstructible.
For maximum degree $3$ or $3$-regular graphs, discussed in
Section~\ref{sec:regular}, much less is known.  It is only known that
$3$-regular graphs are $2$-reconstructible, with no nontrivial upper bounds
known on the maximum reconstructibility.

\section{Trees}\label{sec:trees}

Trees have played a prominent role in the study of reconstruction.  The
original 1957 paper of Kelly~\cite{kel2} showed that trees are
$1$-reconstructible.  Giles~\cite{Giles} showed in 1976 that trees with
at least five vertices are $2$-reconstructible ($P_4$ and $K_{1,3}$ have the
same $2$-deck).  According to N\'ydl~\cite{N01}, the survey of Bondy and
Hemminger~\cite{BH77} reported the existence of a preprint by Giles
proving that sufficiently large trees are $\ell$-reconstructible, but this
was apparently never published and seems to remain open.

N\'ydl gave a lower bound for the threshold $M'_\ell$ such that for
$n\ge M'_\ell$, no two $n$-vertex trees have the same $(n-\ell)$-deck.
Since that paper is somewhat inaccessible (and does not present a full proof),
we give a new short proof here using the third statement of Claim~\ref{twocomp}.
Recall that $S_{a,b,c}$ is the spider with $a+b+c+1$ vertices consisting of
paths of lengths $a$, $b$, and $c$ with a common endpoint.

\begin{theorem}[N\'ydl~\cite{N90}]\label{trees}
The two trees $S_{k-1,k-1,1}$ and $S_{k,k-2,1}$ have the same $k$-deck.
\end{theorem}
\begin{proof}
Let $G=S_{k-1,k-1,1}$ and $H=S_{k,k-2,1}$; both $G$ and $H$ have $2k$ vertices.
We partition the $k$-decks according to the usage of the non-peripheral leaf,
which we call $v$ in each graph.  The portions of the $k$-deck in which $v$
does not appear are the same, since both equal $\cD_k(P_{2k-1})$.  The portions
in which $v$ appears and its neighbor does not are also the same, since they
are the $(k-1)$-decks (plus an isolated vertex) of $P_{k-1}+P_{k-1}$ and
$P_k+P_{k-2}$, which by Claim~\ref{twocomp} are the same.

In the remainder of the cards in the decks, $v$ appears in a nontrivial
component that is a spider $S_{a,b,1}$.  Consider those cards where this
spider takes $a$ vertices from the first leg and $b$ vertices from the second
leg in the original specification of the host trees.  Since $a,b\le k-2$,
such a spider exists in $H$ if and only if it exists in $G$.

For each choice of $(a,b)$, the cards in this portion of the $k$-deck of $G$
consist of the disjoint union of $S_{a,b,1}$ with cards in the 
$(k-a-b-2)$-deck of $P_{k-1-a}+P_{k-1-b}$.  Similarly, in the $k$-deck of $H$
we have the disjoint union of $S_{a,b,1}$ with cards in the $(k-a-b-2)$-deck
of $P_{k-a}+P_{k-2-b}$.  Since $k-a\ge k-a-b-2$ and $k-b-1\ge k-a-b-2$,
Claim~\ref{twocomp} applies to guarantee that these portions of the two
$k$-decks are the same.
\end{proof}

Thus $M'_\ell\ge2\ell+1$.

\begin{conjecture}[N\'ydl~\cite{N90}]\label{Ntree}
$M'_\ell=2\ell+1$.
\end{conjecture}

Note that Conjecture~\ref{Ntree} is not quite the same as $M_\ell(\bT)=2\ell+1$.
N\'ydl required only that no two $n$-vertex trees have the same $(n-\ell)$-deck,
but for $\ell$-reconstructibility there is also the matter of showing that
all possible reconstructions from the deck are trees; that is, showing that
the family of trees is $\ell$-recognizable when $n\ge 2\ell+1$.

An $n$-vertex graph is a tree if and only if it has $n-1$ edges and is 
connected (or has $n-1$ edges and no cycles).  From the $2$-deck, we know the
number of edges.  If Conjecture~\ref{connconj} is true in the stronger form
replacing $2\ell+2$ with $2\ell+1$ for $\ell\ge3$, which Theorem~\ref{conn3}
proves for $\ell=3$, then combined with Conjecture~\ref{Ntree} it would 
imply $M_\ell(\bT)=2\ell+1$.  Indeed, the full strength of 
Conjecture~\ref{connconj} probably is not needed; we only need to know from
the $\FL{n/2}$-deck whether an $n$-vertex graph with $n-1$ edges has a cycle
of length at least $n/2$.

\section{Disconnected and Complete Multipartite Graphs}\label{sec:rpartite}

One of the earliest results on reconstruction, by Kelly~\cite{kel2},
is that disconnected graphs are $1$-reconstructible.  Manvel~\cite{Manvel}
discussed the $\ell$-reconstructibility of disconnected graphs.
We expand on this discussion to obtain a sharp threshold on the size
of components that makes $\ell$-reconstructibility easy.

We first prove what might be called the ``negative'' result.

\begin{proposition}
If graphs with at least $\ell+2$ vertices consisting of a connected graph
and $\ell-1$ isolated vertices are $\ell$-reconstructible, then the original
Reconstruction Conjecture holds.
\end{proposition}
\begin{proof}
We need to prove $1$-reconstructibility when $n\ge3$.  Kelly~\cite{kel2} proved
this for disconnected graphs, so consider a connected $n$-vertex graph $G$.
We are given $\cD_{n-1}(G)$.  By Observation~\ref{k-1}, we also know 
$\cD_{n-i}(G)$ for $2\le i\le \ell$.  Let $G'=G+(\ell-1)K_1$; note that
$G'$ has at least $\ell+2$ vertices.  For $2\le i\le \ell$, let $\cD'_i$
consist of $\CH{\ell-1}{\ell-i}$ copies of $C+(i-1)K_1$ for each occurrence of
$C$ in $\cD_{n-i}(G)$.  Note that
$$
\cD_{n-\ell}(G')=\cD_{n-1}(G)\cup (\cD'_2\cup\cdots\cup\cD'_\ell).
$$
Thus if $G+(\ell-1)K_1$ is $\ell$-reconstructible, then we have determined
$G$ from $\cD_{n-1}(G)$.
\end{proof}

Now consider $n$-vertex graphs whose components all have at most $n-\ell$
vertices.  Manvel~\cite{Manvel} observed that IF it is known that $G$ is a
such a graph, then $G$ is $\ell$-reconstructible.  In fact, we show that such
graphs can be recognized from the $(n-\ell)$-deck.  With Manvel's observation,
this implies that these graphs are in fact $\ell$-reconstructible.

The argument generalizes a proof of the $1$-reconstructibility of disconnected
graphs, involving a counting argument for ordinary graph reconstruction
that was applied by Bondy and Hemminger~\cite{BH78} and originated with
Greenwell and Hemminger~\cite{GH}.  A similar argument to that given here
appears in the ``Main Lemma'' of N\'ydl~\cite{N01}.

We need the basic idea of Kelly's Lemma~\cite{kel2}, which counts the copies of
a given graph $F$ that appear in $G$ by dividing the total number of
appearances of $F$ in the deck by the number of times each copy appears.  We
use an analogue for induced subgraphs and generalize to the $(n-\ell)$-deck.

\begin{lemma}\label{kelly}
If $G$ is an $n$-vertex graph, and $F$ is a graph with at most $n-\ell$
vertices, then the number $s_F(G)$ of occurrences of $F$ as an induced subgraph
of $G$ is $\ell$-reconstructible.
\end{lemma}
\begin{proof}
Let $p=|V(F)|$.  Each induced copy of $F$ appears in $\CH{n-p}\ell$ cards
in $\cD_{n-\ell}(G)$.  Letting $t$ be the total count of all appearances of $F$
as an induced subgraph in all cards in $\cD_{n-\ell}(G)$, we have
$s_F(G)=t\big/\CH{n-p}\ell$.
\end{proof}

\begin{theorem}\label{disconn}
If every connected subgraph of $G$ has at most $n-\ell$ vertices, then
$G$ is $\ell$-reconstructible.
\end{theorem}
\begin{proof}
It suffices to determine, for every connected graph $F$, the number $c_F(G)$ of
components of such a graph $G$ that are isomorphic to $F$.

Let an {\it induced chain} of length $r$ be a list $\VEC F0r$ of connected
induced subgraphs of $G$ such that $F_i$ an induced subgraph of $F_{i+1}$ for
$0\le i<r$.  For any connected induced subgraph $F$ of $G$, let the {\it depth}
of $F$ be the maximum $r$ such that $F$ is the first subgraph in an induced
chain of length $r$.

Since every connected subgraph of $G$ has at most $n-\ell$ vertices,
all connected induced subgraphs of $G$ appear in the deck.  Since we know
all these subgraphs, we can determine all the induced chains, and hence we
know the depth of each connected induced subgraph.

If $F$ has depth $0$, then every induced copy of $F$ in $G$ is a component,
so $c_F(G)=s_F(G)$.  For larger depth, group the induced copies of $F$ by the
unique component of $G$ containing that copy.  Summing over all components of
$G$, we obtain
\begin{equation*}
s_F(G)=\sum_{H} s_F(H)c_H(G).\tag{1}
\end{equation*}

When $s_F(H)\ne0$ and $F\ne H$, every induced chain starting at $H$ can be
augmented by adding $F$ at the beginning, so $H$ has smaller depth than $F$.
Using Lemma~\ref{kelly} to compute $s_F(G)$ and applying the induction
hypothesis to compute values of $c_H(G)$, we now know every quantity in
(1) other than $c_F(G)$ and can solve for $c_F(G)$.
\end{proof}

Here we used the property that for the family $\cF$ of connected graphs, every
member of $\cF$ contained in $G$ has at most $n-\ell$ vertices and belongs to a
unique maximal member of $\cF$ contained in $G$, namely a component.  The
argument applies also to other families $\cF$ that have this property.

For a very special class of disconnected graphs, much stronger results about
reconstructibility are known.  Note first a simple observation, using the
fact that the cards in the $k$-deck determine their complements.

\begin{observation}
A graph $G$ is determined by its $k$-deck if and only if its complement
$\Gb$ is determined by its $k$-deck.
\end{observation}

Hence when we discuss $\ell$-reconstructibility of graphs whose components are
complete graphs, we are also discussing $\ell$-reconstructibility of complete
multipartite graphs.

Let $G$ be a disjoint union of complete graphs.  Membership in this family is
determined by the $3$-deck, since a graph is a disjoint union of complete
graphs if and only if it does not have $P_3$ as an induced subgraph.  When the
largest component has at most $k$ vertices, Theorem~\ref{disconn} implies that
the graph is determined by its $k$-deck (Spinoza and West~\cite{SW} had 
observed that the $(k+1)$-deck suffices).

More interesting is the situation when we bound the number of parts rather
than the size of the parts.

\begin{theorem}[Spinoza--West~\cite{SW}]\label{Trpart}
Every complete $r$-partite graph $G$ is determined by its $(r+1)$-deck
(as are disjoint unions of $r$ complete graphs).
\end{theorem}

The proof of Theorem~\ref{Trpart} is actually algebraic.  For $r\ge2$,
the $3$-deck tells us that $G$ is complete multipartite, and the absence
of $K_{r+1}$ in the deck makes it $r$-partite.  Letting the part-sizes be
$\VEC q1r$, form the polynomial $\PE i1r (x-q_i)$.  The coefficient of
$(-1)^jx^{r-j}$ in the expansion is the number of complete cards in $\cD_i(G)$.
Since $\cD_i(G)$ is determined by $\cD_{r+1}(G)$, we know the polynomial and
can find the roots $\VEC q1r$.

Theorem~\ref{Trpart} is sharp for $r\le2$ and all $n$.  It is immediate that
complete bipartite graphs are not $2$-deck reconstructible, since they are not
determined by their numbers of edges and vertices (the $3$-vertex cards are
not given.  For complete tripartite graphs, the $3$-deck determines that the
graph is a complete multipartite graph, but the following example shows that
that is not sufficient.

\begin{example}[Spinoza--West~\cite{SW}]\label{Erpart}
The complete multipartite graphs $K_{7,4,3}$ and $K_{6,6,1,1}$ have the
same $3$-deck.  It consists of $84$ copies of $K_3$, $240$ copies of 
$P_3$, and $40$ copies of $\Kb_3$.
\end{example}

We expect that diligence will yield more general examples.

\begin{question}
Is it true for all $r\in\NN$ with $r\ge3$ that there are a complete $r$-partite
graph and a complete $(r+1)$-partite graph having the same $r$-deck?
\end{question}

Finally, N\'ydl considered the reconstructibility of disjoint unions of
complete graphs where neither the number of components nor the sizes of the
components are restricted.  We can still recognize from the $3$-deck that our
graph is in this class.

\begin{theorem}[N\'ydl~\cite{N85}]
Let $G$ be an $n$-vertex graph that is a disjoint union of complete graphs.
If $n<k\ln(k/2)$, then $G$ is determined by its $k$-deck.  If
$n=(k+1)2^{k-1}$, then there is such a graph $G$ that is not determined
by its $k$-deck.
\end{theorem}

These bounds are quite far apart, and neither says much about the threshold
number of vertices for $\ell$-reconstructibility of disjoint unions of 
complete graphs (or, equivalently, complete multipartite graphs).
The extremal problem is the following.

\begin{problem}
Determine the maximum $n$ such that every $n$-vertex complete multipartite
graph is determined by its $k$-deck.
\end{problem}

\section{Regular graphs}\label{sec:regular}

As noted above, $1$-reconstructibility of disconnected graphs is easy, but
$2$-reconstructibility of all disconnected graphs implies the Reconstruction
Conjecture.

Similarly, $1$-reconstructibility of regular graphs is easy using the 
$1$-reconstructibility of the degree list.  Motivated by this, at a meeting in
Sanya in 2019 Bojan Mohar asked whether regular graphs are $2$-reconstructible.

Since $1$-regular graphs are determine by their degree lists, they are
determined by their $3$-decks and hence are $(n-3)$-reconstructible.  The
results of Spinoza and West~\cite{SW} described in Section~\ref{sec:maxdeg2}
imply that $2$-regular graphs are $\FL{n/2}$-reconstructible.  Both 
thresholds are sharp.

For $r\ge 3$, $2$-reconstructibility of $r$-regular graphs is not immediate,
even though the degree list is $2$-reconstructible, because we must determine
which of the deficient vertices in a card is adjacent to which of the two
missing vertices.  Nevertheless, the question has been answered for $r=3$.

\begin{theorem}[Kostochka--Nahvi--West--Zirlin~\cite{KNWZ2}]\label{3reg2rec}
Every $3$-regular graph is $2$-reconstructible.
\end{theorem}

Although this result takes considerable effort, it is (we hope) just the
beginning of study in this area.  It would be interesting both to answer
Mohar's question and to determine the maximum reconstructibility for
$3$-regular graphs or for graphs with maximum degree $3$, extending the
results discussed earlier.

\begin{problem}
For each $r\in\NN$ with $r\ge2$, prove that every $r$-regular graph
is $2$-reconstructible.
\end{problem}

\begin{problem}
Show that for each $\ell\geq 1$ there is a threshold $n_\ell$ such that every
$3$-regular graph with at least $n_\ell$ vertices is
$\ell$-reconstructible.
\end{problem}

Although we do not know whether all $r$-regular graphs are $2$-reconstructible,
for those that are not $2$-connected we can say something much stronger.
Note that we are {\it deleting} $r+1$ vertices; the cards have $n-r-1$ vertices.

\begin{theorem}
Every $r$-regular graph $G$ that is not $2$-connected is $(r+1)$-reconstructible.
\end{theorem}
\begin{proof}
If $G$ is disconnected, then every component has at least $r+1$ and hence at
most $n-(r+1)$ vertices.  Thus Theorem~\ref{disconn} applies to make it
$(r+1)$-reconstructible.

Now suppose that $G$ has a cut-vertex.  A subgraph of an $r$-regular graph is
{\em near $r$-regular} if it has exactly one vertex with degree less than $r$.
In every leaf block of $G$, only the cut-vertex of $G$ has degree less than
$r$.  Hence every leaf block is near $r$-regular; furthermore, every near
$r$-regular subgraph of $G$ having no cut-vertex is a leaf block.

Besides the cut-vertex, a leaf block must have at least $r+1$ other vertices;
if only $r$, then $G$ would be $K_{r+1}$.  Since $G$ has at least two leaf
blocks, $G$ has at least $2r+3$ vertices.  Hence the cards in the
$(n-r-1)$-deck have at least $r+2$ vertices, so by
Manvel's result (Theorem~\ref{tlist}) we can reconstruct the degree list.

A $2$-connected $r$-regular graph cannot have a near $r$-regular subgraph $H$
with more than one vertex.  If such $H$ exists, let $x$ be a vertex having
degree $r$ in $H$, and let $y$ be a vertex of $G$ not in $H$.  Since $2$ is
$2$-connected, by Menger's Theorem it has internally disjoint paths from $x$
to $y$.  Such paths must leave $H$ at distinct vertices having degree less than
$r$ in $H$, contradicting that $H$ is near $r$-regular.
Hence we have shown that the class of $r$-regular graphs with connectivity $1$
is $(r+1)$-recognizable.

Since every leaf block omits at least the $r+1$ non-cut-vertices of some other
leaf block, every leaf block has at most $n-(r+1)$ vertices.  
By Observation~\ref{k-1}, we know all the subgraphs of $G$ having at most
$n-(r+1)$ vertices, with their multiplicities.  The near $r$-regular ones 
without cut-vertices are the leaf blocks.  Hence we know all the leaf blocks,
with their multiplicities.

Let $B$ be a leaf block with fewest vertices, and let $s=\C{V(B)}$.
In $\cD_{n-s+1}(G)$ there is a card that has as (leaf) blocks all the leaf
blocks of $G$ other than $B$, and one less leaf block isomorphic to $B$ than
$G$ has.  This card $H$ is near $r$-regular.  Reconstruct $G$ by attaching $B$
at the vertex of $H$ with degree less than $r$.
\end{proof}

\section{Almost All Graphs} \label{sec:almostall}

Using cards not much larger than those that fail to determine connectedness, we
can almost always reconstruct a graph.  Chinn~\cite{Chinn} and
Bollob\'as~\cite{Bollobas} proved that almost all graphs are
$1$-reconstructible.  In fact, this holds also for $\ell$-reconstructibility,
as observed earlier by M\"uller~\cite{Muller}.  The needed tool is that for
almost all graphs, the induced subgraphs with many vertices are pairwise
nonisomorphic and have no nontrivial automorphisms (precise statement below).
We say that a property {\it holds for almost all graphs} if the fraction of
graphs with vertex set $\{1,\ldots,n\}$ for which the property holds tends to
$1$ as $n$ tends to $\infty$.

For $1$-reconstructibility, Chinn proved the following (in a stronger form):

\begin{theorem}[Chinn~\cite{Chinn}]\label{chinn}
If the subgraphs of a graph $G$ obtained by deleting two vertices are pairwise
nonisomorphic, then $G$ is reconstructible.
\end{theorem}

When the subgraphs satisfy this hypothesis, vertex $u$ is identifiable in $G-w$
because it is the only vertex in $G-w$ whose deletion yields a subgraph
obtainable from $G-u$ by deleting one vertex.  From $G-v$ and $G-w$, one can
similarly identify $v$ in $G-w$.  Now one can check whether $u$ and $v$ are
adjacent in $G$ by checking whether $u$ and $v$ are adjacent in $G-w$.
However, since we used both $G-u$ and $G-v$ to determine whether $u$ and $v$
are adjacent, we used all the cards.

\begin{theorem}[Bollob\'as~\cite{Bollobas}]\label{bollob}
For almost every graph, any three cards determine $G$.
\end{theorem}

Under the same hypothesis as in Theorem~\ref{chinn}, Bollob\'as gave a more
careful argument to reconstruct all of $G$ from $G-u$ and $G-v$ except for
determining whether $u$ and $v$ are adjacent.  For that he consulted a third
card, invoking the uniqueness of the graphs in $\cD_{n-3}(G)$ to identify $u$
and $v$ in $G-w$.  However, it seems that uniqueness in $\cD_{n-2}(G)$ suffices
to identify $u$ and $v$ in $G-w$ as discussed above.

Theorem~\ref{bollob} is stronger than saying that {\it some} three cards
determine $G$, which is the meaning of reconstruction number $3$ (two cards can
never determine whether the two deleted vertices are adjacent).  The needed
tool is the next lemma.

\begin{lemma}[M\"uller~\cite{Muller}]\label{lem: almost all}\label{nonisom}
Let $\epsilon$ be a small positive real number.  For almost every graph $G$,
the induced subgraphs with at least $k$ vertices have no nontrival
automorphisms and are pairwise nonisomorphic, where
$k=(1+\epsilon)\frac{|V(G)|}{2}$
\end{lemma}

Via counting arguments, M\"uller showed that graphs with this property are
reconstructible from smaller cards.  Spinoza and West more directly generalized
the combinatorial argument of Bollob\'as, thereby reconstructing the graph from
a small set of cards in the $(k+1)$-deck.  However, one step in their
construction does not work when $\ell=1$.

\begin{theorem}[Spinoza--West~\cite{SW}]\label{lrecon}
For $\ell>1$, if the subgraphs of $G$ obtained by deleting $\ell+1$ vertices
have no nontrival automorphisms and are pairwise nonisomorphic, then $G$ is
$\ell$-reconstructible, using just $\binom{\ell+2}{2}$ cards from the
$(|V(G)|-\ell)$-deck.
\end{theorem}

This not only shows $\ell$-reconstructibility; it also places a bound on the
natural generalization of reconstruction number to the $(n-\ell)$-deck.
Furthermore, asymptotically at least $\CH{n}{\ell+1}(n-\ell-1)^{\CH{\ell+1}2}$
sets of $\CH{\ell+2}2$ cards from $\cD_{n-\ell}(G)$ determine $G$.
The cards are chosen by specifying a fixed set $S$ of $\ell+1$ vertices in $G$
and taking all cards that delete $\ell$ of them, plus for each pair $u,v\in S$
one card obtained by deleting $S-\{u,v\}$ and one vertex outside $S$.

\section{Another Model of Reconstruction}\label{sec:relation}

As mentioned in the introduction, the term ``$k$-reconstructible'' is also used
in another model of reconstruction with different definitions.  Here we explain
the difference in order to reduce confusion.

We use ``digraph'' to mean a general binary relation (no repeated edges).
Two digraphs $D$ and $D'$ on an $n$-element vertex set $V$ are 
{\it $k$-isomorphic} if for every $k$-element subset $X\esub V$, the
subdigraphs of $D$ and $D'$ induced by $X$ are isomorphic.  They are
{\it $(\le k)$-isomorphic} if they are $k'$-isomorphic for all $k'$ with
$1\le k'\le k$.  They are {\it $(-k)$-isomorphic} if they are
$(n-k)$-isomorphic.  A digraph $D$ is {\it $\alpha$-reconstructible}, where
$\alpha\in\{k,\le k,-k\}$, if every digraph $\alpha$-isomorphic to $D$ is
isomorphic to $D$.

These notions were introduced by Fra\"iss\'e~\cite{Fr}, who conjectured
that for sufficiently large $k$ every digraph is $(\le k)$-reconstructible
(and analogously for $m$-ary relations, for each $m$).
The difference between Fra\"iss\'e's model and that of Kelly and Ulam is that
in Fra\"iss\'e's problem we are told the identities of the missing vertices,
but in the problem of Kelly and Ulam we are given only the multiset of
isomorphism types.  The notions coincide for the original conjecture: a graph
(that is, a symmetric digraph) is reconstructible (in the Kelly--Ulam sense) if
and only if it is $(-1)$-reconstructible (in the Fra\"iss\'e sense).
Stockmeyer~\cite{St} showed that general digraphs (in fact, orientations of
complete graphs) are not $(-1)$-reconstructible.

The difference is clear when $k=2$.  Only graphs with at most one edge (and
their complements) are reconstructible from their $2$-decks, but every
symmetric digraph is $2$-reconstructible, since we are told which pairs are
adjacent.  This does not hold for general digraphs; any two orientations of a
complete graph are $2$-isomorphic.

Fra\"iss\'e's conjecture was proved for digraphs (that is, binary relations) by
Lopez~\cite{Lo1,Lo2}, who proved that every digraph is $(\le6)$-reconstructible
(this is sharp).  The theorem was proved independently by Reid and
Thomassen~\cite{RT}, and it also follows from the later characterization of the
non-$(\le k)$-reconstructible digraphs by Boudabbous and Lopez~\cite{BL}.
A history of the topic appears in \cite{BHZ}.

Analogously to Observation~\ref{k-1}, Pouzet showed that if two $n$-vertex
digraphs are $p$-isomorphic, then they are also $q$-isomorphic whenever
$1\le q\le \min\{p,n-p\}$.  With Lopez's Theorem, this implies that every
digraph with at least $11$ vertices is $6$-reconstructible, and every
digraph with at least $12$ vertices is $(-6)$-reconstructible.

\end{document}